\documentclass[a4paper]{article}

\usepackage{geometry}
\usepackage{amsmath,amssymb,amsfonts,amsthm}
\usepackage{latexsym}
\usepackage{indentfirst}
\usepackage{graphicx}
\usepackage[colorlinks=true]{hyperref}
\usepackage{mathrsfs}
\usepackage{chngcntr}
\usepackage{color}

\numberwithin{equation}{section}

\newtheorem{Thm}{Theorem}[section]

\newtheorem{Lem}{Lemma}[section]
\newtheorem{Pro}{Proposition}[section]

\newtheorem{Rek}{Remark}[section]
\newtheorem{Def}{Definition}[section]

\newcommand{\R}{\mathbb{R}}

\title{Normalized solutions for a class of fractional Choquard equations with the HLS lower critical term and a nonlocal perturbation}

\author{Shaoxiong Chen, Vishvesh Kumar, Zhipeng Yang\thanks{Corresponding author:yangzhipeng326@163.com.}, Xi Zhang}

\date{}

\AtEndDocument{%
  \par
  \bigskip
  \bigskip

  \noindent
  \textbf{Shaoxiong Chen:}\\[0.2em]
  \textsc{Department of Mathematics, Yunnan Normal University, Kunming, China}\\[0.3em]
  \textit{E-mail address}: \texttt{gxmail@126.com}\\[1.5em]
  \noindent
  \textbf{Vishvesh Kumar:}\\[0.2em]
  \textsc{Department of Mathematical Sciences, Indian Institute of Technology (BHU), Varanasi, India.}\\[0.3em]
  \textit{E-mail address}: \texttt{vishvesh.mat@iitbhu.ac.in and vishveshmishra@gmail.com}\\[1.5em]
  \noindent
  \textbf{Zhipeng Yang:}\\[0.2em]
  \textsc{Department of Mathematics, Yunnan Normal University, Kunming, China}\\
  \textsc{Yunnan Key Laboratory of Modern Analytical Mathematics and Applications, Kunming, China.}\\[0.3em]
  \textit{E-mail address}: \texttt{yangzhipeng326@163.com}\\[1.5em]
  \noindent
  \textbf{Xi Zhang:}\\[0.2em]
  \textsc{Department of Mathematics, Yunnan Normal University, Kunming, China}\\[0.3em]
  \textit{E-mail address}: \texttt{1359410616@qq.com}
}

\begin{document}

\date{}
\maketitle

\begin{abstract}
In this paper, we study the mass-constrained fractional Choquard equation
\[
\left\{
\begin{array}{l}
(-\Delta)^s u
=
\lambda u
+\alpha \bigl(I_\mu * |u|^{\frac{2N-\mu}{N}}\bigr)|u|^{\frac{2N-\mu}{N}-2}u
+\bigl(I_\mu * |u|^p\bigr)|u|^{p-2}u
\quad \text{in } \R^N,\\[4pt]
\displaystyle \int_{\R^N}|u|^2\,dx=c^2>0,
\end{array}
\right.
\]
where \(N>2s\), \(s\in(0,1)\), \(\mu\in(0,N)\), \(\alpha>0\), and
\[
2+\frac{2s-\mu}{N}\le p<\frac{2N-\mu}{N-2s}.
\]
We first establish a nonexistence result in the \(L^2\)-critical case \(p=2+\frac{2s-\mu}{N}\). Then, in the \(L^2\)-supercritical range, we prove the existence of normalized ground states in two complementary regimes determined by the quantity \(\mathcal{M}_1(c)\). Our approach is based on constrained variational methods, a min-max construction, and refined estimates for the associated fiber maps.


\end{abstract}


\section{Introduction}

In this paper, we consider the fractional Choquard problem
\begin{equation}\label{eq1.1}
\left\{
\begin{array}{l}
(-\Delta)^s u
=
\lambda u
+\alpha \bigl(I_\mu * |u|^{\frac{2N-\mu}{N}}\bigr)|u|^{\frac{2N-\mu}{N}-2}u
+\bigl(I_\mu * |u|^p\bigr)|u|^{p-2}u
\quad \text{in } \R^N,\\[4pt]
\displaystyle \int_{\R^N}|u|^2\,dx=c^2>0,
\end{array}
\right.
\end{equation}
where
\[
N>2s,\qquad s\in(0,1),\qquad \mu\in(0,N),\qquad \alpha>0,\qquad c>0,
\]
and
\[
2+\frac{2s-\mu}{N}\le p<\frac{2N-\mu}{N-2s}.
\]
Here \(\lambda\in\R\) is the Lagrange multiplier associated with the mass constraint, and
\[
(I_\mu*f)(x)=\int_{\R^N}\frac{f(y)}{|x-y|^\mu}\,dy,
\qquad x\in\R^N,
\]
denotes the Riesz-type potential used throughout the paper. The fractional Laplacian is given by
\[
(-\Delta)^s u(x)
=
C_{N,s}\,\mathrm{P.V.}\int_{\R^N}\frac{u(x)-u(y)}{|x-y|^{N+2s}}\,dy,
\]
for suitable functions \(u\), where \(C_{N,s}>0\) is a normalization constant.

Problems involving the fractional Laplacian arise in many areas, including anomalous diffusion, phase transitions, and mathematical physics. Fractional Choquard equations combine the nonlocal dispersion generated by \((-\Delta)^s\) with a Hartree-type nonlocal interaction, and have been studied extensively with respect to existence, multiplicity, regularity, and qualitative properties of solutions. We refer to
\cite{2012DiBSM,2007FrohlichCMP,2000LaskinPLA,2004ApplebaumNAM,2016Molica}
for background on the fractional Laplacian, and to
\cite{2020FilippucciNA,2001LeibAMS,2013MorozJFA,2021HeJDE,2024MengQTDS,2017MaNA,2021BihuaANA,2017BhattaraiJDE}
for general works on Choquard-type equations.

For normalized solutions, a decisive variational viewpoint is to search for critical points of the energy under the prescribed mass constraint
\[
\int_{\R^N}|u|^2\,dx=c^2,
\]
so that the frequency \(\lambda\) appears as a Lagrange multiplier. In this direction, Bartsch and Soave \cite{BartschSoave2017} introduced a natural constraint approach based on the Pohozaev identity, which has become a fundamental tool when the constrained functional is not bounded from below on the whole mass shell.

A particularly relevant line for the present paper is the theory of mixed nonlinearities. In the local case, Soave \cite{SoaveJDE2020} studied the normalized problem
\[
-\Delta u=\lambda u+\mu |u|^{q-2}u+|u|^{p-2}u
\quad \text{in }\R^N,
\qquad
\int_{\R^N}|u|^2\,dx=a^2,
\]
and established existence, nonexistence, and stability or instability properties of normalized ground states under various assumptions on \(q\), \(p\), \(a\), and \(\mu\). In particular, when
\[
2<q\le 2+\frac{4}{N}\le p<2^*,
\]
the two nonlinear terms have different characters with respect to the \(L^2\)-critical exponent, and the geometry of the constrained functional changes substantially. Soave \cite{SoaveJFA2020} also treated the Sobolev critical counterpart
\[
-\Delta u=\lambda u+\mu |u|^{q-2}u+|u|^{2^*-2}u,
\qquad
\int_{\R^N}|u|^2\,dx=a^2,
\]
and obtained a normalized analogue of the Brezis--Nirenberg picture, again showing that the interaction between two terms of different critical nature leads to delicate threshold phenomena.

For normalized Choquard equations, Bartsch, Liu, and Liu \cite{BartschLiuLiu2020} proved the existence of normalized solutions for a broad class of nonlinear Choquard equations by variational methods under the mass constraint. Cingolani and Tanaka \cite{CingolaniTanaka2021} further developed deformation arguments under a prescribed mass condition and established the existence of ground state solutions for nonlinear Choquard equations with prescribed mass. Yuan, Chen, and Tang  \cite{YuanChenTang2020} also studied generalized Choquard equations with prescribed \(L^2\)-norm and obtained positive normalized solutions by a minimax procedure.

At the Hardy--Littlewood--Sobolev lower critical level, several perturbative models have been investigated in recent years. Yao, Chen, R\u{a}dulescu, and Sun \cite{YaoChenRadulescuSun2022} studied normalized solutions for lower critical Choquard equations with a critical Sobolev perturbation,
\[
-\Delta u+\lambda u
=
\gamma (I_\alpha*|u|^{\frac{N+\alpha}{N}})|u|^{\frac{\alpha}{N}-1}u
+\mu |u|^{2^*-2}u,
\qquad
\int_{\R^N}|u|^2\,dx=c^2,
\]
and proved several existence and nonexistence results, including the doubly critical regime. Li, Bao, and Tang \cite{LiBaoTang2023} considered the lower critical Choquard equation with a local perturbation
\[
-\Delta u+\lambda u
=
\gamma (I_\alpha*|u|^{\frac{N+\alpha}{N}})|u|^{\frac{N+\alpha}{N}-2}u
+\mu |u|^{q-2}u,
\qquad
\int_{\R^N}|u|^2\,dx=c^2,
\]
and obtained existence, nonexistence, and multiplicity results together with a non-autonomous extension.

There is also a growing literature for fractional normalized Choquard equations. He, R\u{a}dulescu, and Zou \cite{HeRadulescuZou2022} studied the critical fractional Choquard equation with a local perturbation,
\[
(-\Delta)^s u
=
\lambda u+\mu |u|^{q-2}u
+(I_\alpha*|u|^{2_{\alpha,s}^*})|u|^{2_{\alpha,s}^*-2}u,
\qquad
\int_{\R^N}|u|^2\,dx=a^2,
\]
and established several existence and nonexistence results for normalized ground states under \(L^2\)-subcritical, \(L^2\)-critical, and \(L^2\)-supercritical perturbations. Yu, Tang, and Zhang \cite{YuTangZhang2023} investigated the lower critical fractional Choquard equation with a focusing local perturbation,
\[
(-\Delta)^s u
=
\lambda u+\gamma (I_\alpha*|u|^{1+\frac{\alpha}{N}})|u|^{\frac{\alpha}{N}-1}u
+\mu |u|^{q-2}u,
\qquad
\int_{\R^N}|u|^2\,dx=a^2,
\]
and proved nonexistence, existence, and symmetry of normalized ground states in different perturbative regimes.

The present problem is doubly nonlocal and mixed in a different sense. We consider
\[
(-\Delta)^s u
=
\lambda u
+\alpha \bigl(I_\mu * |u|^{\frac{2N-\mu}{N}}\bigr)|u|^{\frac{2N-\mu}{N}-2}u
+\bigl(I_\mu * |u|^p\bigr)|u|^{p-2}u
\quad \text{in } \R^N,
\qquad
\int_{\R^N}|u|^2\,dx=c^2,
\]
where the first Hartree term is exactly of Hardy--Littlewood--Sobolev lower critical type, while the second Hartree term has exponent \(p\). Thus both nonlinearities are nonlocal, but they play different roles with respect to the mass-preserving scaling. In this sense, the problem is closer in spirit to the mixed-nonlinearity theory developed by Soave, but now in a fractional Choquard framework where both nonlinear terms are of Hartree type and one of them is exactly Hardy--Littlewood--Sobolev lower critical.

The present problem contains the Hardy--Littlewood--Sobolev lower critical exponent
\[
2_{\mu,*}=\frac{2N-\mu}{N},
\]
and the upper Hardy--Littlewood--Sobolev critical exponent
\[
2_{\mu,s}^*=\frac{2N-\mu}{N-2s}.
\]
The lower critical term is accompanied by an additional nonlocal perturbation of order \(p\). Let
\[
S_c=\left\{u\in H^s(\R^N): \|u\|_2^2=c^2\right\}.
\]
The associated constrained energy functional is
\begin{equation}\label{eq1.2}
J_\alpha(u)
=
\frac{1}{2}\|(-\Delta)^{s/2}u\|_2^2
-\frac{\alpha}{2\,2_{\mu,*}}\int_{\R^N}\bigl(I_\mu*|u|^{2_{\mu,*}}\bigr)|u|^{2_{\mu,*}}\,dx
-\frac{1}{2p}\int_{\R^N}\bigl(I_\mu*|u|^p\bigr)|u|^p\,dx,
\end{equation}
defined on \(S_c\).

To describe the scaling properties, for \(u\in S_c\) and \(t\in\R\), we set
\[
u^t(x)=e^{\frac{N}{2}t}u(e^t x).
\]
Then \(u^t\in S_c\), and the lower critical term is invariant under this mass-preserving scaling. Moreover, if we define
\begin{equation}\label{eq1.3}
\gamma_{p,s}=\frac{N(p-2)+\mu}{2sp},
\end{equation}
then
\[
\int_{\R^N}\bigl(I_\mu*|u^t|^p\bigr)|u^t|^p\,dx
=
e^{2p\gamma_{p,s}st}
\int_{\R^N}\bigl(I_\mu*|u|^p\bigr)|u|^p\,dx.
\]
Therefore, we conclude that
\[
2+\frac{2s-\mu}{N}
\]
is exactly the \(L^2\)-critical exponent for the perturbation term, while the regime
\[
2+\frac{2s-\mu}{N}<p<\frac{2N-\mu}{N-2s}
\]
is \(L^2\)-supercritical and HLS-subcritical.

For later use, we introduce
\begin{equation}\label{eq1.4}
\mathcal{M}_1(c)
=
c^{\frac{2p(1-\gamma_{p,s})}{p\gamma_{p,s}-1}}.
\end{equation}
This quantity will determine the two existence regimes in our main results.

Our goal is to prove existence of normalized ground states for \eqref{eq1.1} in the \(L^2\)-supercritical and HLS-subcritical range
\[
2+\frac{2s-\mu}{N}<p<\frac{2N-\mu}{N-2s},
\]
and to show that two complementary existence regimes are determined by the size of \(\mathcal{M}_1(c)\). We also establish a nonexistence result in the \(L^2\)-critical case
\[
p=2+\frac{2s-\mu}{N}.
\]
From the mathematical viewpoint, our contribution may be viewed as a nonlocal counterpart of the mixed-nonlinearity theory for normalized Schr\"odinger equations, but now in a fractional Choquard setting where both nonlinear terms are Hartree-type and one of them is exactly Hardy--Littlewood--Sobolev lower critical.

We now state our main results.

\begin{Thm}\label{Thm1.1}
Assume that
\[
p=2+\frac{2s-\mu}{N}.
\]
If
\[
1-\frac{N C_p\, c^{2\left(\frac{2s-\mu}{N}+1\right)}}{2N+2s-\mu}>0,
\]
where \(C_p\) is the constant given in Lemma~\ref{Lem2.2}, then \(J_\alpha\) has no critical point on \(S_c\). Consequently, problem \eqref{eq1.1} admits no normalized solution.
\end{Thm}

\begin{Rek}\label{Rek1.1}
Theorem~\ref{Thm1.1} yields a nonexistence result in the \(L^2\)-critical case under an explicit small-mass condition. The existence or nonexistence of normalized solutions outside this regime is left open in the present paper.
\end{Rek}

\begin{Thm}\label{Thm1.2}
Assume that
\[
N>2s,
\qquad
2+\frac{2s-\mu}{N}<p<\frac{2N-\mu}{N-2s},
\qquad
\alpha>0,
\qquad
c>0.
\]
Then there exists a constant \(\delta>0\), depending only on the fixed parameters of the problem, such that if
\[
\mathcal{M}_1(c)<\delta,
\]
then problem \eqref{eq1.1} admits a normalized ground state solution \((u_c,\lambda_c)\). Moreover, \(u_c\) can be chosen nonnegative and radially symmetric, and
\[
-\lambda_c\sim c^{\frac{2(1-p)}{p\gamma_{p,s}-1}}.
\]
\end{Thm}

\begin{Thm}\label{Thm1.3}
Assume that
\[
N>2s,
\qquad
2+\frac{2s-\mu}{N}<p<\frac{2N-\mu}{N-2s},
\qquad
\alpha>0,
\qquad
c>0.
\]
Then there exists a constant \(\tau>0\), depending only on the fixed parameters of the problem, such that if
\[
\mathcal{M}_1(c)>\tau,
\]
then problem \eqref{eq1.1} admits a normalized ground state solution \((u_c,\lambda_c)\). Moreover, \(u_c\) can be chosen nonnegative and radially symmetric. In addition,
\[
\left|-\lambda_c c^2-\alpha S_\mu^{-\frac{2N-\mu}{N}}c^{\frac{2(2N-\mu)}{N}}\right|
\lesssim
c^{\frac{2(p\gamma_{p,s}-p)}{p\gamma_{p,s}-1}},
\]
and
\[
m_1(c,\alpha)
+\frac{N\alpha}{2(2N-\mu)}S_\mu^{-\frac{2N-\mu}{N}}c^{\frac{2(2N-\mu)}{N}}
\sim
c^{\frac{2(p\gamma_{p,s}-p)}{p\gamma_{p,s}-1}},
\]
where \(m_1(c,\alpha)\) denotes the ground-state energy level defined later.
\end{Thm}

The remainder of the paper is organized as follows. In Section~2, we introduce the variational framework, recall several preliminary results, and prove the nonexistence result in the \(L^2\)-critical case. In Section~3, we study the fibering maps, establish the minimax characterization of the ground-state level, and derive the compactness properties of the associated Palais--Smale sequences. In Section~4, we obtain the key asymptotic estimates for the energy level and the Lagrange multiplier, and then prove the existence of normalized ground states in the two complementary regimes described in Theorems~\ref{Thm1.2} and \ref{Thm1.3}.

\section{Preliminaries}

This section is devoted to presenting the variational framework and basic tools used in the sequel. We first recall the functional setting and the notion of weak solution to \eqref{eq1.1}.

For \(s\in(0,1)\), the fractional Sobolev space \(H^s(\R^N)\) is defined by
\[
\begin{aligned}
H^s(\R^N)
&=
\Bigl\{
u\in L^2(\R^N):
\frac{u(x)-u(y)}{|x-y|^{\frac{N}{2}+s}}
\in L^2(\R^N\times\R^N)
\Bigr\} \\
&=
\Bigl\{
u\in L^2(\R^N):
\int_{\R^N}(1+|\xi|^{2s})\,|\mathcal{F}(u)(\xi)|^2\,d\xi<\infty
\Bigr\},
\end{aligned}
\]
where \(\mathcal{F}(u)\) denotes the Fourier transform of \(u\). The norm in \(H^s(\R^N)\) is given by
\[
\|u\|_{H^s(\R^N)}
=
\left(
\int_{\R^N}\int_{\R^N}
\frac{|u(x)-u(y)|^2}{|x-y|^{N+2s}}\,dx\,dy
+\int_{\R^N}|u|^2\,dx
\right)^{1/2}.
\]
Moreover, by Propositions 3.4 and 3.6 in \cite{2012DiBSM},
\[
\int_{\R^N}\bigl|(-\Delta)^{\frac{s}{2}}u\bigr|^2\,dx
=
\int_{\R^N}|\xi|^{2s}|\mathcal{F}(u)(\xi)|^2\,d\xi
=
\frac{1}{2}C_{N,s}
\int_{\R^N}\int_{\R^N}
\frac{|u(x)-u(y)|^2}{|x-y|^{N+2s}}\,dx\,dy,
\]
where \(C_{N,s}>0\) depends only on \(N\) and \(s\). Hence we will often use the equivalent norm
\[
\|u\|_{H^s(\R^N)}
=
\left(
\int_{\R^N}|u|^2\,dx
+
\int_{\R^N}\bigl|(-\Delta)^{\frac{s}{2}}u\bigr|^2\,dx
\right)^{1/2}.
\]

We also introduce the homogeneous fractional Sobolev space
\[
\mathcal{D}^{s,2}(\R^N)
=
\Bigl\{
u\in L^{2_s^*}(\R^N):
\int_{\R^N}\bigl|(-\Delta)^{\frac{s}{2}}u\bigr|^2\,dx<\infty
\Bigr\},
\qquad
2_s^*=\frac{2N}{N-2s},
\]
equipped with the norm
\[
\|u\|
=
\left(
\int_{\R^N}\bigl|(-\Delta)^{\frac{s}{2}}u\bigr|^2\,dx
\right)^{1/2}.
\]
In what follows, \(\|\cdot\|\) always denotes this homogeneous norm, while \(\|\cdot\|_{H^s}\) denotes the full \(H^s\)-norm.

We define
\[
H_{\mathrm{rad}}^s(\R^N)
=
\{u\in H^s(\R^N): u \text{ is radial}\},
\qquad
S_{c,\mathrm{rad}}=S_c\cap H_{\mathrm{rad}}^s(\R^N).
\]

\begin{Def}\label{Def2.1}
A function \(u\in H^s(\R^N)\) is called a weak solution of \eqref{eq1.1} if \(u\in S_c\) and there exists \(\lambda\in\R\) such that
\begin{equation}\label{eq2.1}
\begin{aligned}
\int_{\R^N}(-\Delta)^{\frac{s}{2}}u\,(-\Delta)^{\frac{s}{2}}v\,dx
&=
\lambda\int_{\R^N}uv\,dx
+\alpha\int_{\R^N}\bigl(I_\mu*|u|^{\frac{2N-\mu}{N}}\bigr)|u|^{\frac{2N-\mu}{N}-2}uv\,dx \\
&\quad
+\int_{\R^N}\bigl(I_\mu*|u|^p\bigr)|u|^{p-2}uv\,dx,
\qquad \forall v\in H^s(\R^N).
\end{aligned}
\end{equation}
\end{Def}

The associated energy functional \(J_\alpha:H^s(\R^N)\to\R\) is
\begin{equation}\label{eq2.2}
J_\alpha(u)
=
\frac{1}{2}\|u\|^2
-\frac{\alpha N}{2(2N-\mu)}
\int_{\R^N}\bigl(I_\mu*|u|^{\frac{2N-\mu}{N}}\bigr)|u|^{\frac{2N-\mu}{N}}\,dx
-\frac{1}{2p}
\int_{\R^N}\bigl(I_\mu*|u|^p\bigr)|u|^p\,dx.
\end{equation}

We also introduce the Pohozaev functional
\[
P_\alpha(u)
=
s\|u\|^2
-s\gamma_{p,s}
\int_{\R^N}\bigl(I_\mu*|u|^p\bigr)|u|^p\,dx,
\]
where
\[
\gamma_{p,s}=\frac{N(p-2)+\mu}{2ps}.
\]
The Pohozaev manifold associated with \(J_\alpha\) at mass \(c\) is defined by
\[
\mathfrak{P}_{\alpha,c}
=
\{u\in S_c: P_\alpha(u)=0\}.
\]

For \(u\in S_c\) and \(t\in\R\), we introduce the mass-preserving scaling
\[
(t\star u)(x)=e^{\frac{Nt}{2}}u(e^t x),
\qquad x\in\R^N.
\]
Then \(t\star u\in S_c\) for all \(t\in\R\). The associated fibering map is
\[
\begin{aligned}
E_u(t)
&=J_\alpha(t\star u) \\
&=
\frac{e^{2st}}{2}\|u\|^2
-\frac{\alpha N}{2(2N-\mu)}
\int_{\R^N}\bigl(I_\mu*|u|^{\frac{2N-\mu}{N}}\bigr)|u|^{\frac{2N-\mu}{N}}\,dx
-\frac{e^{2p\gamma_{p,s}st}}{2p}
\int_{\R^N}\bigl(I_\mu*|u|^p\bigr)|u|^p\,dx.
\end{aligned}
\]
Hence
\[
E_u'(t)
=
s e^{2st}\|u\|^2
-
s\gamma_{p,s}e^{2p\gamma_{p,s}st}
\int_{\R^N}\bigl(I_\mu*|u|^p\bigr)|u|^p\,dx,
\]
and
\[
E_u''(t)
=
2s^2 e^{2st}\|u\|^2
-
2ps^2\gamma_{p,s}^2 e^{2p\gamma_{p,s}st}
\int_{\R^N}\bigl(I_\mu*|u|^p\bigr)|u|^p\,dx.
\]

\begin{Rek}\label{Rek2.1}
For every \(u\in S_c\), one has
\[
E_u'(0)=P_\alpha(u).
\]
Moreover, for every \(u\in S_c\) and \(t\in\R\),
\[
E_u'(t)=0
\quad\Longleftrightarrow\quad
t\star u\in \mathfrak{P}_{\alpha,c}.
\]
In particular,
\[
\mathfrak{P}_{\alpha,c}
=
\{u\in S_c: E_u'(0)=0\}.
\]
We further decompose
\[
\mathfrak{P}_{\alpha,c}
=
\mathfrak{P}_{\alpha,c}^+
\cup
\mathfrak{P}_{\alpha,c}^-
\cup
\mathfrak{P}_{\alpha,c}^0,
\]
where
\[
\mathfrak{P}_{\alpha,c}^+
=
\{u\in\mathfrak{P}_{\alpha,c}: E_u''(0)>0\},
\]
\[
\mathfrak{P}_{\alpha,c}^-
=
\{u\in\mathfrak{P}_{\alpha,c}: E_u''(0)<0\},
\]
and
\[
\mathfrak{P}_{\alpha,c}^0
=
\{u\in\mathfrak{P}_{\alpha,c}: E_u''(0)=0\}.
\]
\end{Rek}

\begin{Rek}\label{Rek2.2}
If \(u\in S_c\) is a critical point of \(J_\alpha|_{S_c}\), then the associated Pohozaev identity yields
\[
P_\alpha(u)=0.
\]
Therefore every constrained critical point of \(J_\alpha\) on \(S_c\) belongs to \(\mathfrak{P}_{\alpha,c}\). Later we will show that \(\mathfrak{P}_{\alpha,c}\) is a natural constraint for \(J_\alpha\).
\end{Rek}

\begin{Rek}\label{Rek2.3}
For
\[
2_{\mu,*}<r<2_{\mu,s}^*,
\]
one has
\[
r\gamma_{r,s}
\begin{cases}
<1, & \displaystyle 2_{\mu,*}<r<2+\frac{2s-\mu}{N},\\[4pt]
=1, & \displaystyle r=2+\frac{2s-\mu}{N},\\[4pt]
>1, & \displaystyle 2+\frac{2s-\mu}{N}<r<2_{\mu,s}^*.
\end{cases}
\]
\end{Rek}

\begin{Pro}\cite{2023YuDCDS}\label{Pro2.1}
Assume that
\[
p\in\left[\frac{2N-\mu}{N},\,\frac{2N-\mu}{N-2s}\right).
\]
Let \(\{u_n\}\subset H^s(\R^N)\) be such that \(u_n\rightharpoonup u\) in \(H^s(\R^N)\). Then, for any \(\varphi\in H^s(\R^N)\),
\[
\int_{\R^N}\bigl(I_\mu*|u_n|^p\bigr)|u_n|^{p-2}u_n\varphi\,dx
\to
\int_{\R^N}\bigl(I_\mu*|u|^p\bigr)|u|^{p-2}u\varphi\,dx.
\]
\end{Pro}

\begin{Pro}\cite{2026sunCN}\label{Pro2.2}
Let
\[
p\in\left[\frac{2N-\mu}{N},\,\frac{2N-\mu}{N-2s}\right]
\]
and let \(\{u_n\}\) be bounded in \(L^{\frac{2Np}{2N-\mu}}(\R^N)\). If \(u_n\to u\) almost everywhere in \(\R^N\), then
\[
\int_{\R^N}\left(I_\mu*|u_n-u|^p\right)|u_n-u|^p\,dx
=
\int_{\R^N}\left(I_\mu*|u_n|^p\right)|u_n|^p\,dx
-
\int_{\R^N}\left(I_\mu*|u|^p\right)|u|^p\,dx
+o_n(1).
\]
\end{Pro}

\begin{Lem}\cite{2001LeibAMS}\label{Lem2.1}
Let \(r,t>1\) and \(\mu\in(0,N)\) satisfy
\[
\frac{1}{r}+\frac{1}{t}=2-\frac{\mu}{N}.
\]
If \(f\in L^r(\R^N)\) and \(h\in L^t(\R^N)\), then there exists a sharp constant \(C(r,t,\mu,N)>0\) such that
\begin{equation}\label{eq2.3}
\int_{\R^N}\int_{\R^N}
\frac{f(x)h(y)}{|x-y|^\mu}\,dx\,dy
\le
C(r,t,\mu,N)\|f\|_r\|h\|_t.
\end{equation}
\end{Lem}

By Lemma \ref{Lem2.1} and the fractional Sobolev embedding, the functional \(J_\alpha\) is well defined on \(H^s(\R^N)\) and belongs to \(C^1(H^s(\R^N),\R)\).

\begin{Lem}\cite{2018FengJMAA}\label{Lem2.2}
Let \(N>2s\), \(s\in(0,1)\), and
\[
2_{\mu,*}<t<2_{\mu,s}^*,
\]
where
\[
2_{\mu,*}=\frac{2N-\mu}{N},
\qquad
2_{\mu,s}^*=\frac{2N-\mu}{N-2s}.
\]
Then, for every \(u\in H^s(\R^N)\),
\begin{equation}\label{eq2.4}
\int_{\R^N}(I_\mu*|u|^t)|u|^t\,dx
\le
C_t\,\|u\|^{2t\gamma_{t,s}}\|u\|_2^{2t(1-\gamma_{t,s})},
\end{equation}
where
\[
\gamma_{t,s}=\frac{N(t-2)+\mu}{2ts},
\]
and \(C_t>0\) depends only on \(t,s,N,\mu\).
\end{Lem}

\begin{Rek}\label{Rek2.4}
At the lower critical exponent \(t=2_{\mu,*}=\frac{2N-\mu}{N}\), one has the sharp inequality
\begin{equation}\label{eq2.5}
\int_{\R^N}\left(I_\mu*|u|^{\frac{2N-\mu}{N}}\right)|u|^{\frac{2N-\mu}{N}}\,dx
\le
S_\mu^{-\frac{2N-\mu}{N}}
\left(\int_{\R^N}|u|^2\,dx\right)^{\frac{2N-\mu}{N}}
\qquad \text{for all }u\in L^2(\R^N),
\end{equation}
where
\begin{equation}\label{eq2.6}
S_\mu
=
\inf\left\{
\int_{\R^N}|u|^2\,dx:
u\in L^2(\R^N),\
\int_{\R^N}\left(I_\mu*|u|^{2-\frac{\mu}{N}}\right)|u|^{2-\frac{\mu}{N}}\,dx=1
\right\}.
\end{equation}
Moreover, \(S_\mu\) is achieved, and every optimizer is of the form
\begin{equation}\label{eq2.7}
V_{\varepsilon,z}(x)
=
K\left(\frac{\varepsilon}{\varepsilon^2+|x-z|^2}\right)^{\frac{N}{2}}
\end{equation}
for some \(K\in\R\setminus\{0\}\), \(\varepsilon>0\), and \(z\in\R^N\).
\end{Rek}

\begin{Lem}\cite{1993GhoussoubCUP}\label{Lem2.3}
Let \(\varphi\) be a \(C^1\)-functional on a complete connected \(C^1\)-Finsler manifold \(X\) without boundary, and let \(\mathcal{F}\) be a homotopy-stable family of compact subsets of \(X\) with closed boundary \(B\). Set
\[
c=c(\varphi,\mathcal{F})=\inf_{A\in\mathcal{F}}\max_{x\in A}\varphi(x),
\]
and assume that
\[
\sup_B\varphi<c.
\]
Then, for any sequence \(\{A_n\}\subset\mathcal{F}\) such that
\[
\sup_{x\in A_n}\varphi(x)\to c,
\]
there exists a sequence \(\{x_n\}\subset X\) such that
\[
\varphi(x_n)\to c,
\qquad
\|d\varphi(x_n)\|\to 0,
\qquad
\mathrm{dist}(x_n,A_n)\to 0.
\]
Moreover, if \(d\varphi\) is uniformly continuous, then \(x_n\) can be chosen in \(A_n\) for every \(n\).
\end{Lem}

\subsection{The \(L^2\)-critical case}

We now give the proof of Theorem~\ref{Thm1.1}. In this subsection we assume that
\[
p=2+\frac{2s-\mu}{N}.
\]
Then
\[
p\gamma_{p,s}=1,
\qquad
\gamma_{p,s}=\frac{1}{p}=\frac{N}{2N+2s-\mu}.
\]

\begin{proof}[Proof of Theorem~\ref{Thm1.1}]
Let \(u\in S_c\). Since \(p\gamma_{p,s}=1\), the fibering map
\[
E_u(t)=J_\alpha(t\star u)
\]
satisfies
\[
E_u'(t)
=
s e^{2st}\|u\|^2
-
s\gamma_{p,s}e^{2st}
\int_{\mathbb{R}^N}(I_\mu*|u|^p)|u|^p\,dx.
\]
Hence
\[
E_u'(t)
=
s e^{2st}
\left(
\|u\|^2
-
\gamma_{p,s}\int_{\mathbb{R}^N}(I_\mu*|u|^p)|u|^p\,dx
\right).
\]
By Lemma~\ref{Lem2.2},
\[
\int_{\mathbb{R}^N}(I_\mu*|u|^p)|u|^p\,dx
\le
C_p\,\|u\|^{2p\gamma_{p,s}}\|u\|_2^{2p(1-\gamma_{p,s})}.
\]
Since \(p\gamma_{p,s}=1\) and \(\|u\|_2^2=c^2\), this becomes
\[
\int_{\mathbb{R}^N}(I_\mu*|u|^p)|u|^p\,dx
\le
C_p\,\|u\|^2\,c^{2p(1-\gamma_{p,s})}.
\]
Moreover,
\[
2p(1-\gamma_{p,s})
=
2(p-1)
=
2\left(\frac{2s-\mu}{N}+1\right).
\]
Therefore,
\[
E_u'(t)
\ge
s e^{2st}\|u\|^2
\left(
1-\gamma_{p,s}C_p\,c^{2\left(\frac{2s-\mu}{N}+1\right)}
\right).
\]
Using
\[
\gamma_{p,s}=\frac{1}{p}=\frac{N}{2N+2s-\mu},
\]
we obtain
\[
E_u'(t)
\ge
s e^{2st}\|u\|^2
\left(
1-\frac{N C_p}{2N+2s-\mu}\,c^{2\left(\frac{2s-\mu}{N}+1\right)}
\right).
\]
Under the assumption
\[
1-\frac{N C_p}{2N+2s-\mu}\,c^{2\left(\frac{2s-\mu}{N}+1\right)}>0,
\]
it follows that
\[
E_u'(t)>0
\qquad\text{for all }t\in\mathbb{R}.
\]
Thus \(E_u\) is strictly increasing on \(\mathbb{R}\), and in particular it has no critical point.

Now suppose by contradiction that \(u\in S_c\) is a constrained critical point of \(J_\alpha\) on \(S_c\). Since the path
\[
t\mapsto t\star u
\]
lies entirely in \(S_c\), one must have
\[
E_u'(0)=\frac{d}{dt}\Big|_{t=0}J_\alpha(t\star u)=0,
\]
which is impossible because \(E_u'(t)>0\) for every \(t\in\mathbb{R}\).

Hence \(J_\alpha\) has no critical point on \(S_c\). Consequently, problem~\eqref{eq1.1} admits no normalized solution under the stated condition.
\end{proof}

\section{The minimax and compactness lemmas}

Throughout this section, we assume that
\[
N>2s,
\qquad
2+\frac{2s-\mu}{N}<p<2_{\mu,s}^*=\frac{2N-\mu}{N-2s}.
\]

\begin{Lem}\label{Lem3.1}
The following conclusions hold.

\noindent (1) For every \(u\in S_c\), the fibering map \(E_u\) has a unique critical point \(t_u\in\mathbb{R}\), characterized by
\begin{equation}\label{eq3.1}
E_u(t_u)=\max_{t\in\mathbb{R}}E_u(t)
\qquad\text{and}\qquad
t_u\star u\in\mathfrak{P}_{\alpha,c},
\end{equation}
with
\begin{equation}\label{eq3.2}
t_u
=
\frac{1}{2s(p\gamma_{p,s}-1)}
\ln\left(
\frac{\|u\|^2}
{\gamma_{p,s}\int_{\mathbb{R}^N}(I_\mu*|u|^p)|u|^p\,dx}
\right).
\end{equation}
Moreover, the map
\[
S_c\ni u\mapsto t_u\in\mathbb{R}
\]
is of class \(C^1\).

\noindent (2) The functional \(J_\alpha\) is coercive and bounded from below on \(\mathfrak{P}_{\alpha,c}\), and
\begin{equation}\label{eq3.3}
m_1(c,\alpha)
=
\inf_{u\in\mathfrak{P}_{\alpha,c}}J_\alpha(u)
=
\inf_{u\in S_c}\max_{t\in\mathbb{R}}J_\alpha(t\star u).
\end{equation}

\noindent (3) One has
\begin{equation}\label{eq3.4}
\inf_{u\in\mathfrak{P}_{\alpha,c}}\|u\|^2
\ge
\left(
\frac{1}{C_p\gamma_{p,s}}\,c^{2p(\gamma_{p,s}-1)}
\right)^{\frac{1}{p\gamma_{p,s}-1}}.
\end{equation}

\noindent (4) One has
\begin{equation}\label{eq3.5}
m_1(c,\alpha)\lesssim
c^{\frac{2(p\gamma_{p,s}-p)}{p\gamma_{p,s}-1}}.
\end{equation}
\end{Lem}

\begin{proof}
(1) For \(u\in S_c\),
\[
E_u'(t)
=
s e^{2st}\|u\|^2
-
s\gamma_{p,s}e^{2p\gamma_{p,s}st}
\int_{\mathbb{R}^N}(I_\mu*|u|^p)|u|^p\,dx.
\]
Hence \(E_u'(t)=0\) if and only if
\[
e^{2s(1-p\gamma_{p,s})t}
=
\frac{\gamma_{p,s}\int_{\mathbb{R}^N}(I_\mu*|u|^p)|u|^p\,dx}{\|u\|^2}.
\]
Since \(p\gamma_{p,s}>1\), the function
\[
t\mapsto e^{2s(1-p\gamma_{p,s})t}
\]
is strictly decreasing from \(+\infty\) to \(0\). Therefore \(E_u\) has a unique critical point \(t_u\), given by \eqref{eq3.2}. Since \(E_u'(t)>0\) for \(t<t_u\) and \(E_u'(t)<0\) for \(t>t_u\), \eqref{eq3.1} follows. The fact that \(t_u\star u\in\mathfrak{P}_{\alpha,c}\) is equivalent to \(E_u'(t_u)=0\). The \(C^1\)-regularity of \(u\mapsto t_u\) follows from the implicit function theorem, because
\[
E_u''(t_u)
=
2s^2\gamma_{p,s}(1-p\gamma_{p,s})
e^{2p\gamma_{p,s}st_u}
\int_{\mathbb{R}^N}(I_\mu*|u|^p)|u|^p\,dx
<0.
\]

(2) If \(u\in\mathfrak{P}_{\alpha,c}\), then
\[
\|u\|^2
=
\gamma_{p,s}
\int_{\mathbb{R}^N}(I_\mu*|u|^p)|u|^p\,dx.
\]
Hence
\[
\begin{aligned}
J_\alpha(u)
&=
\frac{1}{2}\|u\|^2
-\frac{\alpha N}{2(2N-\mu)}
\int_{\mathbb{R}^N}(I_\mu*|u|^{\frac{2N-\mu}{N}})|u|^{\frac{2N-\mu}{N}}\,dx
-\frac{1}{2p\gamma_{p,s}}\|u\|^2 \\
&=
\left(\frac{1}{2}-\frac{1}{2p\gamma_{p,s}}\right)\|u\|^2
-\frac{\alpha N}{2(2N-\mu)}
\int_{\mathbb{R}^N}(I_\mu*|u|^{\frac{2N-\mu}{N}})|u|^{\frac{2N-\mu}{N}}\,dx.
\end{aligned}
\]
Using \eqref{eq2.5}, we obtain
\[
J_\alpha(u)
\ge
\left(\frac{1}{2}-\frac{1}{2p\gamma_{p,s}}\right)\|u\|^2
-
\frac{\alpha N}{2(2N-\mu)}
S_\mu^{-\frac{2N-\mu}{N}}c^{\frac{2(2N-\mu)}{N}},
\]
which yields that \(J_\alpha\) is coercive and bounded from below on \(\mathfrak{P}_{\alpha,c}\). The identity \eqref{eq3.3} follows from part (1).

(3) If \(u\in\mathfrak{P}_{\alpha,c}\), then by Lemma \ref{Lem2.2},
\[
\|u\|^2
=
\gamma_{p,s}\int_{\mathbb{R}^N}(I_\mu*|u|^p)|u|^p\,dx
\le
C_p\gamma_{p,s}\|u\|^{2p\gamma_{p,s}}c^{2p(1-\gamma_{p,s})}.
\]
Since \(p\gamma_{p,s}>1\), \eqref{eq3.4} follows.

(4) Fix \(v\in S_{1,\mathrm{rad}}\). Set \(v_c=cv\in S_{c,\mathrm{rad}}\). By \eqref{eq3.2},
\[
e^{2st_{v_c}}
=
\left(
\frac{\|v\|^2}
{\gamma_{p,s}\int_{\mathbb{R}^N}(I_\mu*|v|^p)|v|^p\,dx}
\right)^{\frac{1}{p\gamma_{p,s}-1}}
c^{-\frac{2(p-1)}{p\gamma_{p,s}-1}}.
\]
Therefore
\[
\begin{aligned}
m_1(c,\alpha)
&\le
J_\alpha(t_{v_c}\star v_c) \\
&=
\left(\frac{1}{2}-\frac{1}{2p\gamma_{p,s}}\right)
e^{2st_{v_c}}\|v_c\|^2
-\frac{\alpha N}{2(2N-\mu)}
\int_{\mathbb{R}^N}(I_\mu*|v_c|^{\frac{2N-\mu}{N}})|v_c|^{\frac{2N-\mu}{N}}\,dx \\
&\lesssim
c^{\frac{2(p\gamma_{p,s}-p)}{p\gamma_{p,s}-1}}.
\end{aligned}
\]
This proves \eqref{eq3.5}.
\end{proof}

\begin{Lem}\label{Lem3.2}
One has
\begin{equation}\label{eq3.11}
m_1(c,\alpha)
=
\inf_{u\in\mathfrak{P}_{\alpha,c}}J_\alpha(u)
=
\inf_{u\in\mathfrak{P}_{\alpha,c}\cap S_{c,\mathrm{rad}}}J_\alpha(u).
\end{equation}
\end{Lem}

\begin{proof}
The inequality
\[
\inf_{u\in\mathfrak{P}_{\alpha,c}}J_\alpha(u)
\le
\inf_{u\in\mathfrak{P}_{\alpha,c}\cap S_{c,\mathrm{rad}}}J_\alpha(u)
\]
is obvious. We prove the reverse one.

Let \(u\in\mathfrak{P}_{\alpha,c}\), and let \(|u|^*\) be the symmetric decreasing rearrangement of \(|u|\). By the fractional P\'olya--Szeg\H{o} inequality and the Riesz rearrangement inequality,
\[
\bigl\|(-\Delta)^{\frac{s}{2}}|u|^*\bigr\|_2^2
\le
\bigl\|(-\Delta)^{\frac{s}{2}}u\bigr\|_2^2,
\]
and, for \(r\in\left\{\frac{2N-\mu}{N},p\right\}\),
\[
\int_{\mathbb{R}^N}(I_\mu*|u|^r)|u|^r\,dx
\le
\int_{\mathbb{R}^N}(I_\mu*||u|^*|^r)||u|^*|^r\,dx.
\]
Since \(\||u|^*\|_2=\|u\|_2=c\), one has \(|u|^*\in S_{c,\mathrm{rad}}\). Moreover,
\[
J_\alpha(t\star |u|^*)\le J_\alpha(t\star u)
\qquad\text{for every }t\in\mathbb{R}.
\]
Therefore, using Lemma \ref{Lem3.1},
\[
J_\alpha(u)
=
\max_{t\in\mathbb{R}}J_\alpha(t\star u)
\ge
\max_{t\in\mathbb{R}}J_\alpha(t\star |u|^*)
=
J_\alpha(t_{|u|^*}\star |u|^*),
\]
and \(t_{|u|^*}\star |u|^*\in\mathfrak{P}_{\alpha,c}\cap S_{c,\mathrm{rad}}\). Taking the infimum over \(u\in\mathfrak{P}_{\alpha,c}\), we obtain the reverse inequality.
\end{proof}

\begin{Lem}\label{Lem3.3}
For \(u\in S_{c,\mathrm{rad}}\), the map
\[
\Psi_u:T_uS_{c,\mathrm{rad}}\to T_{t_u\star u}S_{c,\mathrm{rad}},
\qquad
\Psi_u(\psi)=t_u\star\psi,
\]
is a linear isomorphism, where
\[
T_uS_{c,\mathrm{rad}}
=
\left\{
\psi\in H^s_{\mathrm{rad}}(\mathbb{R}^N):
\int_{\mathbb{R}^N}u\psi\,dx=0
\right\}.
\]
Its inverse is given by
\[
\Psi_u^{-1}(\varphi)=(-t_u)\star\varphi.
\]
\end{Lem}

\begin{proof}
For \(\psi\in T_uS_{c,\mathrm{rad}}\),
\[
\int_{\mathbb{R}^N}(t_u\star u)(t_u\star\psi)\,dx
=
\int_{\mathbb{R}^N}u\psi\,dx
=
0,
\]
so \(\Psi_u\) is well defined. Linearity is immediate. Moreover,
\[
(-t_u)\star(t_u\star\psi)=\psi,
\qquad
t_u\star((-t_u)\star\varphi)=\varphi.
\]
Hence \(\Psi_u\) is a linear isomorphism with inverse \(\varphi\mapsto (-t_u)\star\varphi\).
\end{proof}

\begin{Lem}\label{Lem3.4}
Define
\[
\widehat{J}_\alpha:S_{c,\mathrm{rad}}\to\mathbb{R},
\qquad
\widehat{J}_\alpha(u)=J_\alpha(t_u\star u).
\]
Then \(\widehat{J}_\alpha\in C^1(S_{c,\mathrm{rad}},\mathbb{R})\), and for every \(u\in S_{c,\mathrm{rad}}\) and \(\varphi\in T_uS_{c,\mathrm{rad}}\),
\[
d\widehat{J}_\alpha(u)[\varphi]
=
dJ_\alpha(t_u\star u)[t_u\star\varphi].
\]
\end{Lem}

\begin{proof}
Consider
\[
\Phi:\mathbb{R}\times S_{c,\mathrm{rad}}\to\mathbb{R},
\qquad
\Phi(t,u)=J_\alpha(t\star u).
\]
By Lemma \ref{Lem3.1}, the map \(u\mapsto t_u\) is \(C^1\), and
\[
\widehat{J}_\alpha(u)=\Phi(t_u,u).
\]
Hence, by the chain rule,
\[
d\widehat{J}_\alpha(u)[\varphi]
=
\partial_t\Phi(t_u,u)\,dt_u[\varphi]
+
\partial_u\Phi(t_u,u)[\varphi].
\]
Since \(\partial_t\Phi(t_u,u)=E_u'(t_u)=0\), it follows that
\[
d\widehat{J}_\alpha(u)[\varphi]
=
\partial_u\Phi(t_u,u)[\varphi]
=
dJ_\alpha(t_u\star u)[t_u\star\varphi].
\]
\end{proof}

\begin{Lem}\label{Lem3.5}
Let \(\mathcal{F}\) be a homotopy-stable family of compact subsets of \(S_{c,\mathrm{rad}}\) with closed boundary \(B\), and assume that
\[
\sup_{u\in B}\widehat{J}_\alpha(u)<e_{\mathcal F},
\qquad
e_{\mathcal F}
=
\inf_{A\in\mathcal F}\max_{u\in A}\widehat{J}_\alpha(u).
\]
Then there exists a Palais--Smale sequence
\[
\{u_n\}\subset\mathfrak{P}_{\alpha,c}\cap S_{c,\mathrm{rad}}
\]
for \(J_\alpha\) restricted to \(S_{c,\mathrm{rad}}\) at level \(e_{\mathcal F}\).
\end{Lem}

\begin{proof}
Let \(\{D_n\}\subset\mathcal F\) be such that
\[
\max_{u\in D_n}\widehat{J}_\alpha(u)\to e_{\mathcal F}.
\]
Define
\[
\eta:[0,1]\times S_{c,\mathrm{rad}}\to S_{c,\mathrm{rad}},
\qquad
\eta(\theta,u)=(\theta t_u)\star u.
\]
By Lemma \ref{Lem3.1}(1), \(\eta\) is continuous and \(\eta(0,u)=u\). Since \(\mathcal F\) is homotopy-stable,
\[
A_n=\eta(1,D_n)=\{t_u\star u:u\in D_n\}\in\mathcal F.
\]
Moreover, \(A_n\subset\mathfrak{P}_{\alpha,c}\cap S_{c,\mathrm{rad}}\), and
\[
\max_{v\in A_n}J_\alpha(v)=\max_{u\in D_n}\widehat{J}_\alpha(u)\to e_{\mathcal F}.
\]

Now apply Lemma \ref{Lem2.3} to the functional \(\widehat{J}_\alpha\) on
\[
X=S_{c,\mathrm{rad}}.
\]
We obtain a sequence \(\{v_n\}\subset S_{c,\mathrm{rad}}\) such that
\[
\widehat{J}_\alpha(v_n)\to e_{\mathcal F},
\qquad
\|d(\widehat{J}_\alpha|_{S_{c,\mathrm{rad}}})(v_n)\|_*\to 0,
\qquad
\operatorname{dist}_{H^s}(v_n,A_n)\to 0.
\]
Set
\begin{equation}\label{eq3.13}
t_n=t_{v_n},
\qquad
u_n=t_n\star v_n.
\end{equation}
Then \(u_n\in\mathfrak{P}_{\alpha,c}\cap S_{c,\mathrm{rad}}\), and
\[
J_\alpha(u_n)=\widehat{J}_\alpha(v_n)\to e_{\mathcal F}.
\]

We claim that there exists \(C>0\) such that
\begin{equation}\label{eq3.14}
e^{-2st_n}=\frac{\|v_n\|^2}{\|u_n\|^2}\le C
\qquad\text{for all }n.
\end{equation}
Indeed, since \(A_n\subset\mathfrak{P}_{\alpha,c}\) and \(J_\alpha\) is coercive on \(\mathfrak{P}_{\alpha,c}\), the sets \(A_n\) are uniformly bounded in \(H^s\). As \(\operatorname{dist}_{H^s}(v_n,A_n)\to 0\), the sequence \(\{v_n\}\) is bounded in \(H^s\). On the other hand, Lemma \ref{Lem3.1}(3) gives a positive lower bound for \(\|u_n\|\). Hence \eqref{eq3.14} follows.

It remains to prove that \(\{u_n\}\) is a Palais--Smale sequence for \(J_\alpha|_{S_{c,\mathrm{rad}}}\). Let \(\psi\in T_{u_n}S_{c,\mathrm{rad}}\). By Lemma \ref{Lem3.3},
\[
(-t_n)\star\psi\in T_{v_n}S_{c,\mathrm{rad}}.
\]
Moreover, by \eqref{eq3.14},
\[
\|(-t_n)\star\psi\|_{H^s}
\le C\|\psi\|_{H^s}.
\]
Using Lemma \ref{Lem3.4}, we infer
\[
\begin{aligned}
\|d(J_\alpha|_{S_{c,\mathrm{rad}}})(u_n)\|_*
&=
\sup_{\substack{\psi\in T_{u_n}S_{c,\mathrm{rad}}\\ \|\psi\|_{H^s}\le 1}}
|dJ_\alpha(u_n)[\psi]| \\
&=
\sup_{\substack{\psi\in T_{u_n}S_{c,\mathrm{rad}}\\ \|\psi\|_{H^s}\le 1}}
|d\widehat{J}_\alpha(v_n)[(-t_n)\star\psi]| \\
&\le
C\,\|d(\widehat{J}_\alpha|_{S_{c,\mathrm{rad}}})(v_n)\|_*
\to 0.
\end{aligned}
\]
Therefore \(\{u_n\}\) is the required Palais--Smale sequence.
\end{proof}

\begin{Lem}\label{Lem3.6}
Let \(\alpha>0\) and \(c>0\). Then there exists a Palais--Smale sequence
\[
\{u_n\}\subset \mathfrak{P}_{\alpha,c}\cap S_{c,\mathrm{rad}}
\]
for \(J_\alpha\) restricted to \(S_{c,\mathrm{rad}}\) at level \(m_1(c,\alpha)\). Moreover, up to a subsequence,
\[
u_n\rightharpoonup u_c \quad\text{in }H^s_{\mathrm{rad}}(\mathbb{R}^N),
\]
for some nontrivial \(u_c\in H^s_{\mathrm{rad}}(\mathbb{R}^N)\). There exists \(\lambda_c<0\) such that \(u_c\) is a weak solution of
\begin{equation}\label{eq3.15}
(-\Delta)^s u
=
\lambda_c u
+\alpha\left(I_\mu*|u|^{\frac{2N-\mu}{N}}\right)|u|^{\frac{2N-\mu}{N}-2}u
+\left(I_\mu*|u|^p\right)|u|^{p-2}u
\qquad\text{in }\mathbb{R}^N.
\end{equation}
\end{Lem}

\begin{proof}
We apply Lemma \ref{Lem3.5} with \(\mathcal F\) equal to the family of singletons in \(S_{c,\mathrm{rad}}\). In this case,
\[
e_{\mathcal F}
=
\inf_{u\in S_{c,\mathrm{rad}}}\widehat{J}_\alpha(u).
\]
By Lemmas \ref{Lem3.1}(2) and \ref{Lem3.2},
\[
e_{\mathcal F}
=
\inf_{u\in S_{c,\mathrm{rad}}}J_\alpha(t_u\star u)
=
\inf_{u\in\mathfrak{P}_{\alpha,c}\cap S_{c,\mathrm{rad}}}J_\alpha(u)
=
m_1(c,\alpha).
\]
Hence there exists a Palais--Smale sequence
\[
\{u_n\}\subset \mathfrak{P}_{\alpha,c}\cap S_{c,\mathrm{rad}}
\]
at level \(m_1(c,\alpha)\).

Since \(J_\alpha\) is coercive on \(\mathfrak{P}_{\alpha,c}\), the sequence \(\{u_n\}\) is bounded in \(H^s(\mathbb{R}^N)\). Therefore, up to a subsequence,
\[
u_n\rightharpoonup u_c \quad\text{in }H^s_{\mathrm{rad}}(\mathbb{R}^N),
\]
\[
u_n\to u_c \quad\text{in }L^r(\mathbb{R}^N)\ \text{for every }2<r<2_s^*,
\]
and
\[
u_n(x)\to u_c(x)\quad\text{for a.e. }x\in\mathbb{R}^N.
\]

By the Lagrange multiplier rule, there exists a sequence \(\{\lambda_n\}\subset\mathbb{R}\) such that
\begin{equation}\label{eq3.16}
\begin{aligned}
&\int_{\mathbb{R}^N}(-\Delta)^{\frac{s}{2}}u_n(-\Delta)^{\frac{s}{2}}v\,dx
-\lambda_n\int_{\mathbb{R}^N}u_n v\,dx \\
&\quad
-\alpha\int_{\mathbb{R}^N}(I_\mu*|u_n|^{\frac{2N-\mu}{N}})|u_n|^{\frac{2N-\mu}{N}-2}u_n v\,dx
-\int_{\mathbb{R}^N}(I_\mu*|u_n|^p)|u_n|^{p-2}u_n v\,dx
=o_n(1)
\end{aligned}
\end{equation}
for every \(v\in H^s_{\mathrm{rad}}(\mathbb{R}^N)\).

Taking \(v=u_n\) in \eqref{eq3.16}, using \(P_\alpha(u_n)=0\), we obtain
\[
\begin{aligned}
\lambda_n c^2
&=
\|u_n\|^2
-\alpha\int_{\mathbb{R}^N}(I_\mu*|u_n|^{\frac{2N-\mu}{N}})|u_n|^{\frac{2N-\mu}{N}}\,dx
-\int_{\mathbb{R}^N}(I_\mu*|u_n|^p)|u_n|^p\,dx
+o_n(1) \\
&=
-\alpha\int_{\mathbb{R}^N}(I_\mu*|u_n|^{\frac{2N-\mu}{N}})|u_n|^{\frac{2N-\mu}{N}}\,dx
+\left(1-\frac{1}{\gamma_{p,s}}\right)\|u_n\|^2
+o_n(1).
\end{aligned}
\]
Since \(p<2_{\mu,s}^*\), one has \(\gamma_{p,s}<1\), hence \(1-\frac{1}{\gamma_{p,s}}<0\). Using Lemma \ref{Lem3.1}(3), we conclude that, up to a subsequence,
\[
\lambda_n\to\lambda_c<0.
\]

We next show that \(u_c\neq 0\). Assume by contradiction that \(u_c=0\). Since
\[
\frac{2Np}{2N-\mu}\in(2,2_s^*),
\]
we have \(u_n\to 0\) in \(L^{\frac{2Np}{2N-\mu}}(\mathbb{R}^N)\). Then, by the HLS inequality,
\[
\int_{\mathbb{R}^N}(I_\mu*|u_n|^p)|u_n|^p\,dx=o_n(1).
\]
Since \(P_\alpha(u_n)=0\), it follows that \(\|u_n\|^2=o_n(1)\), contradicting Lemma \ref{Lem3.1}(3). Therefore \(u_c\neq 0\).

Finally, passing to the limit in \eqref{eq3.16} by Proposition \ref{Pro2.1}, applied with exponents \(p\) and \(\frac{2N-\mu}{N}\), we obtain that \(u_c\) satisfies \eqref{eq3.15} in the weak sense. Since the functional is \(O(N)\)-invariant, by the principle of symmetric criticality, \(u_c\) is a weak solution in \(H^s(\mathbb{R}^N)\).
\end{proof}

The corresponding action functional associated with \eqref{eq3.15} is
\[
\begin{aligned}
I_\alpha(u)
&=
\frac{1}{2}\|u\|^2
-\frac{\lambda_c}{2}\|u\|_2^2
-\frac{\alpha N}{2(2N-\mu)}
\int_{\mathbb{R}^N}(I_\mu*|u|^{\frac{2N-\mu}{N}})|u|^{\frac{2N-\mu}{N}}\,dx \\
&\quad
-\frac{1}{2p}
\int_{\mathbb{R}^N}(I_\mu*|u|^p)|u|^p\,dx
=
J_\alpha(u)-\frac{\lambda_c}{2}\|u\|_2^2.
\end{aligned}
\]
We also define the Nehari manifold
\[
\mathcal{N}_c
=
\left\{
u\in H^s_{\mathrm{rad}}(\mathbb{R}^N)\setminus\{0\}:
N_c(u)=0
\right\},
\]
where
\[
N_c(u)
=
\|u\|^2
-\lambda_c\|u\|_2^2
-\alpha\int_{\mathbb{R}^N}(I_\mu*|u|^{\frac{2N-\mu}{N}})|u|^{\frac{2N-\mu}{N}}\,dx
-\int_{\mathbb{R}^N}(I_\mu*|u|^p)|u|^p\,dx,
\]
and
\[
m_2(c,\alpha)
=
\inf_{u\in\mathcal{N}_c}I_\alpha(u).
\]
By Lemma \ref{Lem3.6}, to conclude that \(u_c\) is a ground state, it remains to prove the strong convergence of the Palais--Smale sequence.

\begin{Lem}\label{Lem3.7}
Let \(\alpha>0\) and \(c>0\). Let \(\{u_n\}\), \(u_c\), and \(\lambda_c\) be given by Lemma~\ref{Lem3.6}. If
\[
m_1(c,\alpha)-\frac{\lambda_c}{2}c^2
<
m_2(c,\alpha)
+
\frac{N-\mu}{2(2N-\mu)}
\alpha^{-\frac{N}{N-\mu}}
(-\lambda_cS_\mu)^{\frac{2N-\mu}{N-\mu}},
\]
then, up to a subsequence,
\[
u_n\to u_c
\qquad\text{strongly in }H^s_{\mathrm{rad}}(\mathbb{R}^N).
\]
\end{Lem}

\begin{proof}
Set
\[
w_n=u_n-u_c.
\]
Since \(u_n\rightharpoonup u_c\) in \(H^s_{\mathrm{rad}}(\mathbb{R}^N)\), one has
\begin{equation}\label{eq3.17}
\|u_n\|^2=\|w_n\|^2+\|u_c\|^2+o_n(1),
\qquad
\|u_n\|_2^2=\|w_n\|_2^2+\|u_c\|_2^2+o_n(1).
\end{equation}
By Proposition~\ref{Pro2.2}, applied with exponents \(p\) and \(\frac{2N-\mu}{N}\),
\begin{equation}\label{eq3.18}
\int_{\mathbb{R}^N}(I_\mu*|u_n|^p)|u_n|^p\,dx
=
\int_{\mathbb{R}^N}(I_\mu*|w_n|^p)|w_n|^p\,dx
+
\int_{\mathbb{R}^N}(I_\mu*|u_c|^p)|u_c|^p\,dx
+o_n(1),
\end{equation}
and
\begin{equation}\label{eq3.19}
\begin{aligned}
\int_{\mathbb{R}^N}(I_\mu*|u_n|^{\frac{2N-\mu}{N}})|u_n|^{\frac{2N-\mu}{N}}\,dx
&=
\int_{\mathbb{R}^N}(I_\mu*|w_n|^{\frac{2N-\mu}{N}})|w_n|^{\frac{2N-\mu}{N}}\,dx \\
&\quad
+
\int_{\mathbb{R}^N}(I_\mu*|u_c|^{\frac{2N-\mu}{N}})|u_c|^{\frac{2N-\mu}{N}}\,dx
+o_n(1).
\end{aligned}
\end{equation}

Since \(u_c\) is a weak solution of \eqref{eq3.15}, one has \(u_c\in\mathcal{N}_c\). Moreover, by the Pohozaev identity,
\[
P_\alpha(u_c)=0.
\]
Using \(P_\alpha(u_n)=0\), \eqref{eq3.17}, and \eqref{eq3.18}, we get
\[
P_\alpha(w_n)=o_n(1).
\]
On the other hand, since
\[
\frac{2Np}{2N-\mu}\in(2,2_s^*),
\]
the radial compact embedding gives \(w_n\to 0\) in \(L^{\frac{2Np}{2N-\mu}}(\mathbb{R}^N)\), hence
\[
\int_{\mathbb{R}^N}(I_\mu*|w_n|^p)|w_n|^p\,dx=o_n(1).
\]
Therefore
\begin{equation}\label{eq3.20}
\|w_n\|^2=o_n(1).
\end{equation}

Next, from the Euler--Lagrange equation for \(u_n\) and the convergence \(\lambda_n\to\lambda_c\), we have
\[
N_c(u_n)=o_n(1).
\]
Since \(N_c(u_c)=0\), \eqref{eq3.17}, \eqref{eq3.18}, and \eqref{eq3.19} yield
\[
N_c(w_n)=o_n(1).
\]
Using \eqref{eq3.20}, we infer
\[
-\lambda_c\|w_n\|_2^2
-
\alpha
\int_{\mathbb{R}^N}(I_\mu*|w_n|^{\frac{2N-\mu}{N}})|w_n|^{\frac{2N-\mu}{N}}\,dx
=
o_n(1).
\]

Passing to a subsequence if necessary, we may assume that
\[
-\lambda_c\|w_n\|_2^2\to l
\qquad\text{and}\qquad
\alpha\int_{\mathbb{R}^N}(I_\mu*|w_n|^{\frac{2N-\mu}{N}})|w_n|^{\frac{2N-\mu}{N}}\,dx\to l
\]
for some \(l\ge 0\).

By \eqref{eq2.5}, either \(l=0\) or
\[
l
\ge
\alpha^{-\frac{N}{N-\mu}}
(-\lambda_cS_\mu)^{\frac{2N-\mu}{N-\mu}}.
\]

If \(l=0\), then \(\|w_n\|_2\to 0\). Together with \eqref{eq3.20}, this gives
\[
u_n\to u_c
\qquad\text{strongly in }H^s_{\mathrm{rad}}(\mathbb{R}^N).
\]

Assume now that
\[
l
\ge
\alpha^{-\frac{N}{N-\mu}}
(-\lambda_cS_\mu)^{\frac{2N-\mu}{N-\mu}}.
\]
Since \(u_c\in\mathcal{N}_c\), we have \(I_\alpha(u_c)\ge m_2(c,\alpha)\). Using \eqref{eq3.17}, \eqref{eq3.18}, \eqref{eq3.19}, and \eqref{eq3.20}, we compute
\[
\begin{aligned}
m_1(c,\alpha)-\frac{\lambda_c}{2}c^2
&=
I_\alpha(u_n)+o_n(1) \\
&=
I_\alpha(u_c)
-\frac{\lambda_c}{2}\|w_n\|_2^2
-\frac{\alpha N}{2(2N-\mu)}
\int_{\mathbb{R}^N}(I_\mu*|w_n|^{\frac{2N-\mu}{N}})|w_n|^{\frac{2N-\mu}{N}}\,dx
+o_n(1) \\
&\ge
m_2(c,\alpha)
+
\left(
\frac{1}{2}-\frac{N}{2(2N-\mu)}
\right)l
+o_n(1) \\
&=
m_2(c,\alpha)
+
\frac{N-\mu}{2(2N-\mu)}\,l
+o_n(1) \\
&\ge
m_2(c,\alpha)
+
\frac{N-\mu}{2(2N-\mu)}
\alpha^{-\frac{N}{N-\mu}}
(-\lambda_cS_\mu)^{\frac{2N-\mu}{N-\mu}},
\end{aligned}
\]
which contradicts the assumption of the lemma. Therefore \(l=0\), and the strong convergence follows.
\end{proof}

\section{Existence of ground state solutions}

In this section we establish the key estimates for \(m_1(c,\alpha)\), \(\lambda_c\), and \(m_2(c,\alpha)\), and then prove Theorems~\ref{Thm1.2} and \ref{Thm1.3}.

Throughout this section, we assume that
\[
N>2s,
\qquad
2+\frac{2s-\mu}{N}<p<2_{\mu,s}^*=\frac{2N-\mu}{N-2s},
\qquad
\alpha>0,
\qquad
c>0.
\]

We also set
\[
a=\frac{2(p\gamma_{p,s}-p)}{p\gamma_{p,s}-1},
\qquad
b=\frac{2(2N-\mu)}{N},
\qquad
A_\mu=\frac{\alpha N}{2(2N-\mu)}S_\mu^{-\frac{2N-\mu}{N}}.
\]
Since
\[
2+\frac{2s-\mu}{N}<p<2_{\mu,s}^*,
\]
one has
\[
\frac{1}{p}<\gamma_{p,s}<1,
\qquad
p\gamma_{p,s}>1,
\qquad
a<0<b.
\]

\begin{Lem}\label{Lem4.1}
Let \(\{u_n\}\subset \mathfrak{P}_{\alpha,c}\cap S_{c,\mathrm{rad}}\) be the Palais--Smale sequence obtained in Lemma~\ref{Lem3.6}. Then, for \(n\) sufficiently large,
\begin{equation}\label{eq4.1}
\|u_n\|^2\sim c^{a},
\end{equation}
and
\begin{equation}\label{eq4.2}
m_1(c,\alpha)+A_\mu c^{b}\sim c^{a}.
\end{equation}
\end{Lem}

\begin{proof}
Since \(u_n\in \mathfrak{P}_{\alpha,c}\), one has
\[
\|u_n\|^2
=
\gamma_{p,s}\int_{\mathbb{R}^N}(I_\mu*|u_n|^p)|u_n|^p\,dx.
\]
Hence
\[
\begin{aligned}
J_\alpha(u_n)
&=
\frac{1}{2}\|u_n\|^2
-\frac{\alpha N}{2(2N-\mu)}
\int_{\mathbb{R}^N}(I_\mu*|u_n|^{\frac{2N-\mu}{N}})|u_n|^{\frac{2N-\mu}{N}}\,dx
-\frac{1}{2p\gamma_{p,s}}\|u_n\|^2 \\
&=
\left(\frac{1}{2}-\frac{1}{2p\gamma_{p,s}}\right)\|u_n\|^2
-\frac{\alpha N}{2(2N-\mu)}
\int_{\mathbb{R}^N}(I_\mu*|u_n|^{\frac{2N-\mu}{N}})|u_n|^{\frac{2N-\mu}{N}}\,dx.
\end{aligned}
\]
Since \(J_\alpha(u_n)\to m_1(c,\alpha)\), by \eqref{eq2.5} we obtain
\begin{equation}\label{eq4.3}
m_1(c,\alpha)+A_\mu c^b
\ge
\left(\frac{1}{2}-\frac{1}{2p\gamma_{p,s}}\right)\|u_n\|^2+o_n(1).
\end{equation}
Using Lemma~\ref{Lem3.1}(3), this gives
\[
m_1(c,\alpha)+A_\mu c^b\gtrsim c^a.
\]

For the reverse estimate, let \(\phi\in S_{1,\mathrm{rad}}\) be a fixed optimizer in \eqref{eq2.5}, and set \(v_c=c\phi\in S_{c,\mathrm{rad}}\). Then
\[
\int_{\mathbb{R}^N}(I_\mu*|v_c|^{\frac{2N-\mu}{N}})|v_c|^{\frac{2N-\mu}{N}}\,dx
=
S_\mu^{-\frac{2N-\mu}{N}}c^b.
\]
By Lemma~\ref{Lem3.1}(1),
\[
m_1(c,\alpha)\le J_\alpha(t_{v_c}\star v_c).
\]
Therefore
\[
m_1(c,\alpha)+A_\mu c^b
\le
J_\alpha(t_{v_c}\star v_c)+A_\mu c^b
=
\left(\frac{1}{2}-\frac{1}{2p\gamma_{p,s}}\right)e^{2st_{v_c}}\|v_c\|^2.
\]
By \eqref{eq3.2},
\[
e^{2st_{v_c}}\|v_c\|^2\sim c^a.
\]
Hence
\[
m_1(c,\alpha)+A_\mu c^b\lesssim c^a.
\]
This proves \eqref{eq4.2}. Combining \eqref{eq4.2} with \eqref{eq4.3}, we obtain
\[
\|u_n\|^2\lesssim c^a.
\]
Together with Lemma~\ref{Lem3.1}(3), this yields \eqref{eq4.1}.
\end{proof}

\begin{Lem}\label{Lem4.2}
Let \(\{u_n\}\), \(u_c\), and \(\lambda_c\) be given by Lemma~\ref{Lem3.6}. Then there exists \(\delta>0\) such that the condition of Lemma~\ref{Lem3.7} holds whenever
\[
\mathcal M_1(c)<\delta.
\]
Moreover,
\begin{equation}\label{eq4.4}
-\lambda_c\sim c^{a-2}
=
c^{\frac{2(1-p)}{p\gamma_{p,s}-1}}.
\end{equation}
\end{Lem}

\begin{proof}
Since \(\mathcal M_1(c)=c^{-a}\) and \(a<0\), the condition \(\mathcal M_1(c)<\delta\) means that \(c\) is sufficiently small.

Taking \(v=u_n\) in \eqref{eq3.16}, using \(P_\alpha(u_n)=0\), we obtain
\begin{equation}\label{eq4.5}
\lambda_n c^2
=
\left(1-\frac{1}{\gamma_{p,s}}\right)\|u_n\|^2
-\alpha\int_{\mathbb{R}^N}(I_\mu*|u_n|^{\frac{2N-\mu}{N}})|u_n|^{\frac{2N-\mu}{N}}\,dx
+o_n(1).
\end{equation}
By Lemma~\ref{Lem4.1},
\[
\|u_n\|^2\sim c^a.
\]
On the other hand, by \eqref{eq2.5},
\[
\int_{\mathbb{R}^N}(I_\mu*|u_n|^{\frac{2N-\mu}{N}})|u_n|^{\frac{2N-\mu}{N}}\,dx
\le
S_\mu^{-\frac{2N-\mu}{N}}c^b.
\]
Since \(a<0<b\), one has \(c^b=o(c^a)\) as \(c\to0\). Hence \eqref{eq4.5} gives
\[
-\lambda_n c^2\sim c^a.
\]
Passing to the limit, we obtain
\[
-\lambda_c c^2\sim c^a,
\]
which yields \eqref{eq4.4}.

In particular, since \(a-2<0\), there exists \(\delta_0>0\) such that
\[
\mathcal M_1(c)<\delta_0
\quad\Longrightarrow\quad
-\lambda_c\ge 1.
\]

We now verify the condition in Lemma~\ref{Lem3.7}. For \(u\in\mathcal N_c\), define
\[
A_p(u)
=
\int_{\mathbb{R}^N}(I_\mu*|u|^p)|u|^p\,dx,
\qquad
A_*(u)
=
\int_{\mathbb{R}^N}(I_\mu*|u|^{\frac{2N-\mu}{N}})|u|^{\frac{2N-\mu}{N}}\,dx,
\]
and
\[
B(u)=\|u\|^2-\lambda_c\|u\|_2^2.
\]
Since \(u\in\mathcal N_c\), one has
\[
B(u)=A_p(u)+\alpha A_*(u).
\]
Moreover, if \(\mathcal M_1(c)<\delta_0\), then \(-\lambda_c\ge 1\), and therefore
\[
B(u)=\|u\|^2+(-\lambda_c)\|u\|_2^2\ge \|u\|^2+\|u\|_2^2.
\]
In particular,
\[
\|u\|^2\le B(u),
\qquad
\|u\|_2^2\le B(u).
\]

By Lemma~\ref{Lem2.2},
\[
A_p(u)
\le
C_p\|u\|^{2p\gamma_{p,s}}\|u\|_2^{2p(1-\gamma_{p,s})}
\le
C_p B(u)^{p\gamma_{p,s}}.
\]
By \eqref{eq2.5},
\[
A_*(u)
\le
S_\mu^{-\frac{2N-\mu}{N}}\|u\|_2^{\frac{2(2N-\mu)}{N}}
\le
S_\mu^{-\frac{2N-\mu}{N}} B(u)^{\frac{2N-\mu}{N}}.
\]
Hence
\[
B(u)
=
A_p(u)+\alpha A_*(u)
\le
C_1 B(u)^{p\gamma_{p,s}}
+
C_2 B(u)^{\frac{2N-\mu}{N}},
\]
where \(C_1,C_2>0\) are independent of \(u\) and of \(c\), provided \(\mathcal M_1(c)<\delta_0\).

Since
\[
p\gamma_{p,s}>1,
\qquad
\frac{2N-\mu}{N}>1,
\]
there exists \(\rho>0\), independent of \(u\) and \(c\), such that
\[
0<t<\rho
\quad\Longrightarrow\quad
C_1 t^{p\gamma_{p,s}-1}+C_2 t^{\frac{N-\mu}{N}}<1.
\]
Therefore \(B(u)\ge \rho\) for every \(u\in\mathcal N_c\).

Next, using \(B(u)=A_p(u)+\alpha A_*(u)\), we compute
\[
\begin{aligned}
I_\alpha(u)
&=
\frac{1}{2}B(u)
-\frac{1}{2p}A_p(u)
-\frac{\alpha N}{2(2N-\mu)}A_*(u) \\
&=
\left(\frac{1}{2}-\frac{1}{2p}\right)A_p(u)
+
\frac{N-\mu}{2(2N-\mu)}\alpha A_*(u).
\end{aligned}
\]
Hence
\[
I_\alpha(u)
\ge
\min\left\{
\frac{1}{2}-\frac{1}{2p},
\frac{N-\mu}{2(2N-\mu)}
\right\}
\bigl(A_p(u)+\alpha A_*(u)\bigr).
\]
That is,
\[
I_\alpha(u)\ge C_0 B(u),
\]
where
\[
C_0=
\min\left\{
\frac{1}{2}-\frac{1}{2p},
\frac{N-\mu}{2(2N-\mu)}
\right\}>0.
\]
Since \(B(u)\ge \rho\), we conclude that
\begin{equation}\label{eq4.6}
m_2(c,\alpha)
=
\inf_{u\in\mathcal N_c}I_\alpha(u)
\ge
C_0\rho
= \eta >0
\end{equation}
whenever \(\mathcal M_1(c)<\delta_0\).

Finally, by Lemma~\ref{Lem4.1} and \eqref{eq4.4},
\[
m_1(c,\alpha)-\frac{\lambda_c}{2}c^2\lesssim c^a.
\]
On the other hand,
\[
\alpha^{-\frac{N}{N-\mu}}(-\lambda_c S_\mu)^{\frac{2N-\mu}{N-\mu}}
\sim
c^{(a-2)\frac{2N-\mu}{N-\mu}}
=
c^{\frac{2(1-p)}{p\gamma_{p,s}-1}\cdot\frac{2N-\mu}{N-\mu}}.
\]
Since
\[
(a-2)\frac{2N-\mu}{N-\mu}<a,
\]
the right-hand side dominates \(c^a\) as \(c\to0\). Therefore, using \eqref{eq4.6}, we can choose \(0<\delta\le \delta_0\) such that
\[
m_1(c,\alpha)-\frac{\lambda_c}{2}c^2
<
m_2(c,\alpha)
+
\frac{N-\mu}{2(2N-\mu)}
\alpha^{-\frac{N}{N-\mu}}
(-\lambda_c S_\mu)^{\frac{2N-\mu}{N-\mu}}
\]
whenever \(\mathcal M_1(c)<\delta\). This is exactly the condition of Lemma~\ref{Lem3.7}.
\end{proof}

\begin{Pro}\label{Pro4.1}
The set \(\mathfrak{P}_{\alpha,c}\) is a \(C^1\) submanifold of \(S_c\) of codimension \(1\). Moreover, it is a natural constraint for \(J_\alpha\) on \(S_c\).
\end{Pro}

\begin{proof}
Define
\[
G:S_c\to\mathbb{R},
\qquad
G(u)=P_\alpha(u).
\]
Then
\[
\mathfrak{P}_{\alpha,c}=G^{-1}(0).
\]
Let \(u\in \mathfrak{P}_{\alpha,c}\), and set
\[
\zeta_u=\frac{d}{dt}\Big|_{t=0}(t\star u)\in T_uS_c.
\]
Since \(u\in \mathfrak{P}_{\alpha,c}\), one has
\[
E_u''(0)=dG(u)[\zeta_u].
\]
Moreover,
\[
E_u''(0)
=
2s^2\|u\|^2-2ps^2\gamma_{p,s}^2\int_{\mathbb{R}^N}(I_\mu*|u|^p)|u|^p\,dx.
\]
Using \(P_\alpha(u)=0\), this becomes
\[
E_u''(0)=2s^2(1-p\gamma_{p,s})\|u\|^2<0.
\]
Hence \(dG(u)\neq 0\) on \(T_uS_c\), so \(0\) is a regular value of \(G\). Therefore \(\mathfrak{P}_{\alpha,c}\) is a \(C^1\) submanifold of \(S_c\) of codimension \(1\).

Now let \(u\in \mathfrak{P}_{\alpha,c}\) be a critical point of \(J_\alpha|_{\mathfrak{P}_{\alpha,c}}\). Since \(\mathfrak{P}_{\alpha,c}\subset S_c\) is given by the single constraint \(G(u)=0\), there exists \(\nu\in\mathbb{R}\) such that
\[
d(J_\alpha|_{S_c})(u)=\nu\, dG(u).
\]
Evaluating both sides at \(\zeta_u\in T_uS_c\), we obtain
\[
0=d(J_\alpha|_{S_c})(u)[\zeta_u]
=
\nu\, dG(u)[\zeta_u]
=
\nu\, E_u''(0).
\]
Since \(E_u''(0)\neq 0\), it follows that \(\nu=0\). Hence
\[
d(J_\alpha|_{S_c})(u)=0.
\]
Therefore \(\mathfrak{P}_{\alpha,c}\) is a natural constraint for \(J_\alpha\) on \(S_c\).
\end{proof}

\begin{proof}[Proof of Theorem~\ref{Thm1.2}]
By Lemmas~\ref{Lem3.6}, \ref{Lem3.7}, and \ref{Lem4.2}, there exists
\[
u_c\in H^s_{\mathrm{rad}}(\mathbb{R}^N)\setminus\{0\}
\]
such that, up to a subsequence,
\[
u_n\to u_c
\qquad\text{strongly in }H^s_{\mathrm{rad}}(\mathbb{R}^N).
\]
Since \(u_n\in \mathfrak P_{\alpha,c}\cap S_{c,\mathrm{rad}}\) for every \(n\), the strong convergence yields
\[
u_c\in \mathfrak P_{\alpha,c}\cap S_{c,\mathrm{rad}}.
\]
Moreover,
\[
J_\alpha(u_c)=\lim_{n\to\infty}J_\alpha(u_n)=m_1(c,\alpha).
\]
Hence \(u_c\) is a minimizer of \(J_\alpha\) on \(\mathfrak P_{\alpha,c}\cap S_{c,\mathrm{rad}}\). By Lemma~\ref{Lem3.2}, it is in fact a minimizer of \(J_\alpha\) on \(\mathfrak P_{\alpha,c}\).

Since
\[
\|\,|u_c|\,\|=\|u_c\|,
\qquad
\||u_c|\|_2=\|u_c\|_2,
\]
and the nonlocal terms depend only on \(|u_c|\), we have
\[
|u_c|\in \mathfrak P_{\alpha,c}\cap S_{c,\mathrm{rad}}
\]
and
\[
J_\alpha(|u_c|)=J_\alpha(u_c)=m_1(c,\alpha).
\]
Thus we may assume that \(u_c\) is nonnegative and radially symmetric.

By Proposition~\ref{Pro4.1}, \(\mathfrak P_{\alpha,c}\) is a natural constraint for \(J_\alpha\) on \(S_c\). Therefore \(u_c\) is a constrained critical point of \(J_\alpha\) on \(S_c\), and \((u_c,\lambda_c)\) is a nonnegative radially symmetric normalized ground state solution of \eqref{eq1.1}. The asymptotic formula for \(\lambda_c\) follows from \eqref{eq4.4}.
\end{proof}

\begin{Lem}\label{Lem4.3}
Let \(\{u_n\}\), \(u_c\), and \(\lambda_c\) be given by Lemma~\ref{Lem3.6}. Then there exists \(\tau>0\) such that the condition of Lemma~\ref{Lem3.7} holds whenever
\[
\mathcal M_1(c)>\tau.
\]
Moreover,
\begin{equation}\label{eq4.7}
\left|
-\lambda_c c^2
-\alpha S_\mu^{-\frac{2N-\mu}{N}}c^b
\right|
\lesssim c^a.
\end{equation}
\end{Lem}

\begin{proof}
Since \(\mathcal M_1(c)=c^{-a}\) and \(a<0\), the condition \(\mathcal M_1(c)>\tau\) means that \(c\) is sufficiently large.

Set
\[
Q_n
=
\int_{\mathbb{R}^N}(I_\mu*|u_n|^{\frac{2N-\mu}{N}})|u_n|^{\frac{2N-\mu}{N}}\,dx.
\]
From \(P_\alpha(u_n)=0\) and \(J_\alpha(u_n)\to m_1(c,\alpha)\), we have
\[
m_1(c,\alpha)
=
\left(\frac{1}{2}-\frac{1}{2p\gamma_{p,s}}\right)\|u_n\|^2
-\frac{\alpha N}{2(2N-\mu)}Q_n
+o_n(1).
\]
Hence
\[
m_1(c,\alpha)+A_\mu c^b
=
\left(\frac{1}{2}-\frac{1}{2p\gamma_{p,s}}\right)\|u_n\|^2
+\frac{\alpha N}{2(2N-\mu)}
\left(
S_\mu^{-\frac{2N-\mu}{N}}c^b-Q_n
\right)
+o_n(1).
\]
By Lemma~\ref{Lem4.1}, both \(m_1(c,\alpha)+A_\mu c^b\) and \(\|u_n\|^2\) are of order \(c^a\). Therefore, for \(n\) sufficiently large,
\begin{equation}\label{eq4.8}
Q_n
=
S_\mu^{-\frac{2N-\mu}{N}}c^b
+
O(c^a).
\end{equation}

Now, taking \(v=u_n\) in \eqref{eq3.16}, using \(P_\alpha(u_n)=0\), we get
\[
-\lambda_n c^2
=
\alpha Q_n
+\left(\frac{1}{\gamma_{p,s}}-1\right)\|u_n\|^2
+o_n(1).
\]
Combining this with Lemma~\ref{Lem4.1} and \eqref{eq4.8}, we infer
\[
-\lambda_n c^2
=
\alpha S_\mu^{-\frac{2N-\mu}{N}}c^b
+
O(c^a).
\]
Passing to the limit, we obtain \eqref{eq4.7}.

We next prove a uniform positive lower bound for \(m_2(c,\alpha)\) in the present large-mass regime. By \eqref{eq4.7},
\[
-\lambda_c c^2
=
\alpha S_\mu^{-\frac{2N-\mu}{N}}c^b
+
O(c^a).
\]
Since \(a<b\), it follows that
\[
-\lambda_c c^2\sim c^b.
\]
Hence
\[
-\lambda_c\sim c^{b-2}.
\]
Because
\[
b=\frac{2(2N-\mu)}{N}>2,
\]
we obtain \(-\lambda_c\to\infty\) as \(c\to\infty\). Therefore there exists \(\tau_0>0\) such that
\[
\mathcal M_1(c)>\tau_0
\quad\Longrightarrow\quad
-\lambda_c\ge 1.
\]

For \(u\in\mathcal N_c\), define
\[
A_p(u)
=
\int_{\mathbb{R}^N}(I_\mu*|u|^p)|u|^p\,dx,
\qquad
A_*(u)
=
\int_{\mathbb{R}^N}(I_\mu*|u|^{\frac{2N-\mu}{N}})|u|^{\frac{2N-\mu}{N}}\,dx,
\]
and
\[
B(u)=\|u\|^2-\lambda_c\|u\|_2^2.
\]
Since \(u\in\mathcal N_c\), one has
\[
B(u)=A_p(u)+\alpha A_*(u).
\]
Moreover, if \(\mathcal M_1(c)>\tau_0\), then \(-\lambda_c\ge 1\), and therefore
\[
B(u)=\|u\|^2+(-\lambda_c)\|u\|_2^2\ge \|u\|^2+\|u\|_2^2.
\]
In particular,
\[
\|u\|^2\le B(u),
\qquad
\|u\|_2^2\le B(u).
\]

By Lemma~\ref{Lem2.2},
\[
A_p(u)
\le
C_p\|u\|^{2p\gamma_{p,s}}\|u\|_2^{2p(1-\gamma_{p,s})}
\le
C_p B(u)^{p\gamma_{p,s}}.
\]
By \eqref{eq2.5},
\[
A_*(u)
\le
S_\mu^{-\frac{2N-\mu}{N}}\|u\|_2^{\frac{2(2N-\mu)}{N}}
\le
S_\mu^{-\frac{2N-\mu}{N}} B(u)^{\frac{2N-\mu}{N}}.
\]
Hence
\[
B(u)
=
A_p(u)+\alpha A_*(u)
\le
C_1 B(u)^{p\gamma_{p,s}}
+
C_2 B(u)^{\frac{2N-\mu}{N}},
\]
where \(C_1,C_2>0\) are independent of \(u\) and \(c\), provided \(\mathcal M_1(c)>\tau_0\).

Since
\[
p\gamma_{p,s}>1,
\qquad
\frac{2N-\mu}{N}>1,
\]
there exists \(\rho>0\), independent of \(u\) and \(c\), such that
\[
B(u)\ge \rho
\qquad\text{for all }u\in\mathcal N_c
\]
whenever \(\mathcal M_1(c)>\tau_0\).

Next, using \(B(u)=A_p(u)+\alpha A_*(u)\), we compute
\[
\begin{aligned}
I_\alpha(u)
&=
\frac{1}{2}B(u)
-\frac{1}{2p}A_p(u)
-\frac{\alpha N}{2(2N-\mu)}A_*(u) \\
&=
\left(\frac{1}{2}-\frac{1}{2p}\right)A_p(u)
+
\frac{N-\mu}{2(2N-\mu)}\alpha A_*(u).
\end{aligned}
\]
Hence
\[
I_\alpha(u)\ge C_0 B(u),
\]
where
\[
C_0=
\min\left\{
\frac{1}{2}-\frac{1}{2p},
\frac{N-\mu}{2(2N-\mu)}
\right\}>0.
\]
Therefore,
\[
m_2(c,\alpha)
=
\inf_{u\in\mathcal N_c}I_\alpha(u)
\ge
C_0\rho
= \eta >0
\]
whenever \(\mathcal M_1(c)>\tau_0\).

We now verify the condition in Lemma~\ref{Lem3.7}. By \eqref{eq4.2} and \eqref{eq4.7},
\[
m_1(c,\alpha)-\frac{1}{2}\lambda_c c^2
=
\frac{N-\mu}{2(2N-\mu)}
\alpha S_\mu^{-\frac{2N-\mu}{N}}c^b
+
O(c^a).
\]
On the other hand, \eqref{eq4.7} gives
\[
-\lambda_c c^2
=
\alpha S_\mu^{-\frac{2N-\mu}{N}}c^b\bigl(1+O(c^{a-b})\bigr).
\]
Since \(a-b<0\), one has \(c^{a-b}\to0\) as \(c\to\infty\). Therefore,
\[
-\lambda_c
=
\alpha S_\mu^{-\frac{2N-\mu}{N}}c^{b-2}\bigl(1+O(c^{a-b})\bigr),
\]
and consequently
\[
\alpha^{-\frac{N}{N-\mu}}
(-\lambda_c S_\mu)^{\frac{2N-\mu}{N-\mu}}
=
\alpha S_\mu^{-\frac{2N-\mu}{N}}c^b
+
O(c^a).
\]
It follows that
\[
\frac{N-\mu}{2(2N-\mu)}
\alpha^{-\frac{N}{N-\mu}}
(-\lambda_c S_\mu)^{\frac{2N-\mu}{N-\mu}}
-
\left(
m_1(c,\alpha)-\frac{1}{2}\lambda_c c^2
\right)
=
O(c^a).
\]
Since \(a<0\), this error tends to \(0\) as \(c\to\infty\). Using the above lower bound
\[
m_2(c,\alpha)\ge \eta>0,
\]
we can choose \(\tau\ge \tau_0\) large enough such that
\[
m_1(c,\alpha)-\frac{\lambda_c}{2}c^2
<
m_2(c,\alpha)
+
\frac{N-\mu}{2(2N-\mu)}
\alpha^{-\frac{N}{N-\mu}}
(-\lambda_c S_\mu)^{\frac{2N-\mu}{N-\mu}}
\]
whenever \(\mathcal M_1(c)>\tau\). This is exactly the condition of Lemma~\ref{Lem3.7}.
\end{proof}

\begin{proof}[Proof of Theorem~\ref{Thm1.3}]
By Lemmas~\ref{Lem3.6}, \ref{Lem3.7}, and \ref{Lem4.3}, there exists
\[
u_c\in H^s_{\mathrm{rad}}(\mathbb{R}^N)\setminus\{0\}
\]
such that, up to a subsequence,
\[
u_n\to u_c
\qquad\text{strongly in }H^s_{\mathrm{rad}}(\mathbb{R}^N).
\]
Since \(u_n\in \mathfrak P_{\alpha,c}\cap S_{c,\mathrm{rad}}\) for every \(n\), the strong convergence yields
\[
u_c\in \mathfrak P_{\alpha,c}\cap S_{c,\mathrm{rad}}.
\]
Moreover,
\[
J_\alpha(u_c)=\lim_{n\to\infty}J_\alpha(u_n)=m_1(c,\alpha).
\]
Hence \(u_c\) is a minimizer of \(J_\alpha\) on \(\mathfrak P_{\alpha,c}\cap S_{c,\mathrm{rad}}\). By Lemma~\ref{Lem3.2}, it is in fact a minimizer of \(J_\alpha\) on \(\mathfrak P_{\alpha,c}\).

Since
\[
\|\,|u_c|\,\|=\|u_c\|,
\qquad
\||u_c|\|_2=\|u_c\|_2,
\]
and the nonlocal terms depend only on \(|u_c|\), we have
\[
|u_c|\in \mathfrak P_{\alpha,c}\cap S_{c,\mathrm{rad}}
\]
and
\[
J_\alpha(|u_c|)=J_\alpha(u_c)=m_1(c,\alpha).
\]
Thus we may assume that \(u_c\) is nonnegative and radially symmetric.

By Proposition~\ref{Pro4.1}, \(\mathfrak P_{\alpha,c}\) is a natural constraint for \(J_\alpha\) on \(S_c\). Therefore \(u_c\) is a constrained critical point of \(J_\alpha\) on \(S_c\), and \((u_c,\lambda_c)\) is a nonnegative radially symmetric normalized ground state solution of \eqref{eq1.1}.

The expansion
\[
\left|
-\lambda_c c^2-\alpha S_\mu^{-\frac{2N-\mu}{N}}c^b
\right|
\lesssim c^a
\]
follows from \eqref{eq4.7}, while \eqref{eq4.2} gives
\[
m_1(c,\alpha)+A_\mu c^b\sim c^a.
\]
This completes the proof.
\end{proof}

\section*{Acknowledgments}
We would like to thank the anonymous referee for his/her careful readings of our manuscript and the useful comments. 

\medskip
{\bf Funding:} This work is supported by National Natural Science Foundation of China (12301145, 12261107, 12561020) and Yunnan Fundamental Research Projects (202301AU070144, 202401AU070123).

\medskip
{\bf Author Contributions:} All the authors wrote the main manuscript text together and these authors contributed equally to this work.

\medskip
{\bf Data availability:}  Data sharing is not applicable to this article as no new data were created or analyzed in this study.

\medskip
{\bf Conflict of Interests:} The authors declare that there is no conflict of interest.

\bibliographystyle{plain}
\bibliography{ref}

\end{document}